\documentclass[11pt, reqno]{amsart}

\usepackage{graphicx}
\usepackage{amsmath}
\usepackage{amssymb}
\usepackage{amsfonts}
\usepackage{amsthm}
\usepackage[all]{xy}
\usepackage{enumerate}
\usepackage{subfig}
\usepackage[latin1]{inputenc} 
\usepackage{algpseudocode}
\usepackage{color}
\usepackage{enumerate}
\usepackage{longtable}
\usepackage{array}
\usepackage{float}
\usepackage{mathdots}
\usepackage{yhmath}
\usepackage{MnSymbol}
\usepackage{footnote}
\restylefloat{table}
\usepackage{stmaryrd}
\usepackage{hyperref}
\usepackage{fullpage}
\usepackage{lipsum}

\newtheorem{theorem}{Theorem}[section]
\newtheorem{definition}[theorem]{Definition}

\newtheorem{proposition}[theorem]{Proposition}

\newtheorem{lemma}[theorem]{Lemma}

\newtheorem{remark}[theorem]{Remark}

\title{Twists of non-hyperelliptic curves}

\author{Elisa Lorenzo Garc\'ia}
\address{Universiteit Leiden\\ 
Mathematisch Instituut\\
Niels Bohrweg 1\\
2333 CA Leiden (The Netherlands)}
\email{e.lorenzo.garcia@math.leidenuniv.nl}

 \makeatletter
\let\@wraptoccontribs\wraptoccontribs
\makeatother

\thanks{}
\subjclass[2010]{11G30, 12F12, 14H99}
\keywords{}

\begin{document}


\begin{abstract} In this paper we show a method for computing the set of twists of a non-singular projective curve defined over an arbitrary (perfect) field $k$. The method is based on  a correspondence between twists and solutions to a Galois embedding problem. When in addition, this curve is non-hyperelliptic we show how to compute equations for the twists. If $k=\mathbb{F}_q$ the method then becomes an algorithm, since in this case, the Galois embedding problems that appear are known how to be solved. As an example we compute the set of twists of the non-hyperelliptic genus $6$ curve $x^7-y^3z^4-z^7=0$ when we consider it defined over a number field such that $[k(\zeta_{21}):k]=12$. For each twist equations are exhibited.
\end{abstract}


\maketitle

\section{Introduction}

The study of twists of curves can be a very useful tool for understanding some arithmetic problems. For example, it has been proved to be really helpful for exploring the Sato-Tate conjecture \cite{FKRS}, \cite{FLS}, \cite{FS2}, \cite{FS3},\cite{thesis}. As well as for solving some Diophantine equations \cite{PSS} or computing $\mathbb{Q}-$curves realizing certain Galois representations \cite{BFGL}, \cite{diff}.  

The twists of curves of genus $\leq 2$ are well-known.
While the genus $0$ and $1$ cases date back to a long time ago \cite{Sil}, the genus $2$
case is due to the work of Cardona and Quer over number fields \cite{Cart}, \cite{Carp}, and to Cardona over finite fields \cite{CarFF}. All the 
genus $0$, $1$ or $2$ curves are hyperelliptic (at least in the sense that they are not non-hyperelliptic, since genus $0$ and $1$ curves are not usually called hyperelliptic). However, for genus greater than $2$ 
almost all the curves are non-hyperelliptic. Only few twists of genus $3$ curves over number fields have been previously computed \cite{diff}, \cite{PSS}. Over finite fields, more twists of genus $3$ has been computed \cite{MT10}, but, in this case, equations are not given.

We devote the present paper to show a method for computing twists of smooth curves of genus greater than $0$, and in
the particular case of non-hyperelliptic curves we show how to compute equations for the twists. The method is not completely original since it is based on well-known results, but as far as we know this is the first time that all the strategies used for computing twists are joined together and all the gaps are filled in order to produce a systematic method. In particular, when the field of definition of the curve has characteristic different from zero, the method gives rise to an algorithm.

In a forthcoming paper \cite{Lor14}, this method will be useful for computing the twists of all non-hyperelliptic genus $3$ curves defined over any number field.

\subsection{Outline} The structure of this paper is as follows. Section \ref{11} establishes a correspondence between the set of twists of any smooth and irreducible genus $g>0$ curve $C$ defined over a perfect field $k$ and the set of solutions to a Galois embedding problem, see Theorem \ref{main}. In Section \ref{12}, we show how to compute equations of the twists in the particular case in which the curve $C$ is non-hyperelliptic. We do this by studying the action of the Galois group of a certain extension of the field of definition of the curve $C$, in the vector space of regular differentials $\Omega^1(C)$. Section \ref{13} describes in detail the method obtained for computing the twists of non-hyperelliptic curves. First step is computing a canonical model of the curve. The second one is posing the corresponding Galois embedding problem, whose solutions are in bijection with the set of twists, and solving it. In general, if $k$ is a number field, there is not known method for solving a Galois embedding problem over $k$, and this step should be treated case-by-case. We compute the solutions to an infinity family of such problems in Proposition \ref{q}. Nevertheless, if $k$ is a finite field, any Galois embedding problem over $k$ is known how to be solved (e.g. \cite[Chapter 1]{Ser}). The third and last step is computing equations for the twists. Finally, in Section \ref{14} we illustrate the method by computing all the twists of the non-hyperelliptic genus $6$ curve $x^7-y^3z^4-z^7=0$ when it is considered to be defined over a number field such that $[k(\zeta_{21}):k]=12$.

\subsection{Notation}
We now fix some notation and conventions that will be valid through the paper. For any field $F$, we denote by $\bar{F}$ an algebraic closure of $F$, and by $G_F$ the absolute Galois group $\operatorname{Gal}(\bar{F}/F)$. We recurrently consider the action of $G_F$ on several sets, and this action is in general denoted by left exponentiation. For a field $F$, let $\operatorname{GL}_n(F)$ (resp. $\operatorname{PGL}_n(F)$) be the ring of $n$ by $n$ invertible matrices with coefficients in F (resp. that are projective).

By $k$ we always mean a perfect field. All field extensions of $k$ that we consider are contained in a fixed algebraic closure $\bar{k}$. We write $\zeta_n$ to refer to a primitive $n-$th root of unity in $\bar{k}$. When $k$ is a number field, we denote by $\mathcal{O}_k$ the ring of integers of $k$. 

Given a projective, smooth and geometrically irreducible curve $C/k$ we denote by $\operatorname{Aut}(C)$ the grup of automorphisms of $C$ defined over $\bar{k}$. By $K$ we denote the minimal extension $K/k$ where all the automorphism of $C$ can be defined. The $k-$vector space of regular differentials of $C$ is denoted by $\Omega^1(C)$.

When we work with groups, we usually use the SmallGroup Library-GAP \cite{GAP}. Where the group $<N,r>$ denotes the group of order $N$ that appears in the $r-$th position in such library. By $\text{ID}(G)$, we mean the corresponding GAP notation for the group $G$.

\subsection{Aknowledgments} The author would like to thank to Joan-Carales Lario for bringing this problem to her attention and to Francesc Fit\'e for a careful reading of the manuscript and useful comments and suggestions.

\section{Galois embedding problems}
\label{11}

Let $k$ be a perfect field and $C/k$ be a projective curve of genus $g>0$. Let us denote by $K$ the minimal field where all the automorphisms of $C$ can be defined. Let us define the twisting group $\Gamma := \operatorname{Aut}(C)\rtimes \operatorname{Gal}(K/k)$, where $\operatorname{Gal}(K/k)$ acts naturally on $\operatorname{Aut}(C)$, and the multiplication rule is $(\alpha,\sigma)(\beta,\tau)=(\alpha\,^{\sigma}\beta,\sigma\tau)$ \cite{FL}. 

Let us define the following sets:
\begin{equation}\label{e3}
\operatorname{Twist}_k(C):=\left\{ C'/k\operatorname{curve}\mid\exists\:\overline{k}\operatorname{-isomorphism}\:\phi\colon C'\to C\right\} /k\operatorname{-isomorphism},
\end{equation}
\begin{equation}\label{e1}
\operatorname{H}^{1}(G_k,\operatorname{Aut}(C)):=\left\{ \xi\colon G_k \to \operatorname{Aut}(C)\,\,\text{continuous} \mid \xi_{\sigma\tau}=\xi_{\sigma}{}^{\sigma}\xi_{\tau}  \right\}/\sim ,
\end{equation}
where the topology in $G_k$ is the profinite one, and we consider the discrete topology in $\operatorname{Aut}(C)$. Two cocycles are cohomologous $\xi\sim\xi '$, if and only if, there is $\varphi\in\operatorname{Aut}(C)$ such that $\xi_{\sigma}'=\varphi\cdot\xi_{\sigma}\cdot{}^{\sigma}\varphi^{-1}$. We also define
\begin{equation}\label{e2}
\widetilde{\operatorname{Hom}}(G_k,\Gamma):=\left\{  \Psi\colon G_k\to\Gamma\mid\Psi\:\operatorname{epi}_{2}-\operatorname{morphism}  \right\}/\sim ,
\end{equation}
the meaning of $\operatorname{epi}_{2}-\operatorname{morphism}$ is that $\Psi$ is a continuous group homomorphism (again with the profinite and discrete topologies respectively) such that the composition $\pi_2 \cdot \Psi :\, G_k\to\Gamma\to\operatorname{Gal}(K/k)$ is surjective where $\pi_2:\,\Gamma\to\operatorname{Gal}(K/k)$ (resp. $\pi_1$) is the natural projection on the second (resp. first) component of the elements of $\Gamma$. We say that $\Psi\sim\Psi'$ are equivalent if there is $(\varphi,1)\in\Gamma$ such that $\Psi_{\sigma}'=(\varphi,1)\Psi_{\sigma}(\varphi,1)^{-1}$.  

\begin{definition} With notation above, we say that $L$ is the splitting field of the twist $\phi:\,C'\rightarrow C$, if $L$ is the minimal field where, for all $\alpha\in\text{Aut}(C)$, the isomorphisms $\alpha\circ \phi$ are defined. Similarly, we define the splitting field of a cocycle $\xi$ as the field $L$ that satisfies the condition $$\text{Gal}(\bar{k}/L)=\bigcup_{\xi\sim\xi'}\text{Ker}(\xi').$$ Since $\xi$ is continuous, $L$ is well-defined.

 For an element $\Psi\in\widetilde{\operatorname{Hom}}(G_k,\Gamma)$, we define its splitting field as the field $L$ such that $\text{Gal}(\bar{k}/L)=\text{Ker}(\Psi)$.
\end{definition}

Notice that the previous splitting fields are all of them finite extensions of $k$ since we are considering curves of genus greater than $0$, and then the group $\text{Aut}(C)$ is finite.

\begin{theorem}\label{main} There are natural one-to-one correspondences between the following three sets:
$$
\operatorname{Twist}_k(C)\longrightarrow
\operatorname{H}^{1}(G_k,\operatorname{Aut}(C))\longrightarrow
\widetilde{\operatorname{Hom}}(G_k,\Gamma)$$

These correspondences send $\phi$ to $\xi_{\sigma}=\phi\cdot{}^{\sigma}\phi^{-1}$, and $\xi$ to $\Psi_{\sigma}=(\xi_{\sigma},\overline{\sigma})$, where $\overline{\sigma}$ denotes the projection of $ \sigma\in G_k$ onto $\operatorname{Gal}(K/k)$. Moreover, the splitting fields of elements in these three sets are preserved by these correspondences.
\end{theorem}

\begin{proof} The correspondence between the set (\ref{e3}) of twists $\operatorname{Twist}_k(C)$ and the first Galois cohomology set (\ref{e1}) is well-known and can be found in \cite[$X.2$ Theorem $2.2$]{Sil}. The statement about the splitting fields follows by definition. So, it only remains to prove that the map  between the sets (\ref{e1}) and (\ref{e2}) is a correspondence. Let us first prove that they are well-defined. Clearly, given $\xi\in\operatorname{Twist}_k(C)$, we have that $\Psi$ defined by $\Psi_{\sigma}:=(\xi_{\sigma},\overline{\sigma})$ defines an element in $\widetilde{\operatorname{Hom}}(G_k,\Gamma)$. Conversely, given an element $\Psi\in\widetilde{\operatorname{Hom}}(G_k,\Gamma)$, we have that $\xi:=\pi_1(\Psi)$ defines an element in $\operatorname{Twist}_k(C)$. Finally, it is a straightforward computation to check that this two maps are one inverse to the other and that they preserve the equivalence relations defined in both sets.
\end{proof}
\begin{remark} Notice that any element $\Psi\in\widetilde{\operatorname{Hom}}(G_k,\Gamma)$ can be reinterpreted as a solution to the following Galois embedding problem:

\begin{align}
\xymatrix{
        &            &                   & G_k \ar@{-->>}_-{\Psi}[dl] \ar@{->>}[d]   &    \\
1 \ar[r] &  \operatorname{Aut}(C)\,\ar@{^{(}->}[r]   & \Gamma \ar[r]_-{\pi}  & \operatorname{Gal}(K/k)  \ar[r] &  1
} \label{gep}
\end{align}
Reciprocally, every solution $\Psi$ of the above embedding problem is an element in $\widetilde{\operatorname{Hom}}(G_k,\Gamma)$ and gives rise to a twist of $C$.
In order to keep track of the equivalence classes of twists we must here consider two solutions $\Psi$ and $\Psi'$ equivalent only under the restricted conjugations allowed in the definition of the set $\widetilde{\operatorname{Hom}}(G_k,\Gamma)$, that is slightly different from the standard one \cite[Section $9.4$]{probe}.
\end{remark}

\section{Equations of the twists}
\label{12}

First of all, remark that a twist is not a curve, it is an equivalence class of curves, so when we say that we compute equation for a twist, what we mean is that we compute equations for some particular curve in the equivalence class. Secondly, notice that a curve can have different models, and a particular model for a non-hyperelliptic curve is its canonical model. The method that we present in this section, it is a method for computing the canonical model of a curve in the equivalence class of a twist defined by a cocycle.

This method is a generalization of the one used by Fern\'andez, Gonz\'alez and Lario \cite{diff}. They used it for computing equations of twists of some particular non-hyperelliptic genus $3$ curves, special case for which the canonical model coincides with the plane model.

Notice that, in our context, finding equations for a twist that is given by a cocycle $\xi\in\operatorname{H}^{1}(G_k,\operatorname{Aut}(C))$, is actually equivalent to computing an inverse map for the correspondence in Theorem \ref{main}
$$
\operatorname{Twist}_k(C)\longrightarrow
\operatorname{H}^{1}(G_k,\operatorname{Aut}(C)).
$$

Let $\Omega^1(C)$ be the $k-$vector space of regular differentials of $C$. Let $\omega_1,...,\omega_g$ be a basis of $\Omega^1(C)$, where $g$ is the genus of $C$ (the existence of such a basis can be deduced from the fact that there always exists a canonical divisor defined over the definition field $k$ of the curves, which is a consequence of \cite[II, Lemma 5.8.1]{Sil}). Given a cocycle $\xi:\, \text{G}_k\to \text{Aut}(C)$ and its splitting field $L$, we consider the extension of scalars $\Omega^1_L(C)=\Omega^1(C)\otimes_{k}L$ which is a $k-$vector space of dimension $g[L:k]$. We can then see the elements of  $\Omega_{L}^{1}(C)$ as sums $\sum \lambda_{i}\omega_{i}$ where $\lambda_{i}\in L$. For every $\sigma\in\operatorname{Gal}(L/k)$, we consider the twisted action on $\Omega_L^1(C)$ defined as follows:
$$
(\sum\lambda_{i}\omega_{i})^{\sigma}_{\xi}:=\sum{}^{\sigma}\lambda_{i}\xi_{\sigma}^{*-1}(\omega_{i}).
$$
Here, $\xi_{\sigma}^{*}\in\operatorname{End}_K(\Omega^1(C))$ denotes the pull-back of $\xi_{\sigma}=\phi\cdot^{\sigma}\phi^{-1}\in\operatorname{Aut}_K(C)$.
One readily checks that
$$
\rho_{\xi}\colon\operatorname{Gal}(L/k)\to\operatorname{GL}(\Omega^1_L(C)),\:\:\:\rho_{\xi}(\sigma)(\omega):=\omega_{\xi}^{\sigma}
$$ 
is a $k-$linear representation. Indeed, since $\xi_{\sigma\tau}^{*}=\,^{\sigma}\xi_{\tau}^{*}\cdot\xi^{*}_{\sigma}$, we have

\begin{align}
\rho_{\xi}(\sigma\tau)(\sum\lambda_{i}\omega_{i})&=\sum{}^{\sigma\tau}\lambda_{i}\xi_{\sigma\tau}^{*-1}(\omega_{i}) \notag\\
&=\sum{}^{\sigma\tau}\lambda_{i}\xi_{\sigma}^{*-1}\cdot^{\sigma}\xi_{\tau}^{*-1}(\omega_{i})\notag \\
&= \rho_{\xi}(\sigma)(\sum{}^{\tau}\lambda_{i}\xi_{\tau}^{*-1}(\omega_{i}))\notag \\
&= \rho_{\xi}(\sigma)\rho_{\xi}(\tau)(\sum\lambda_{i}\omega_{i}).\notag
\end{align}

\begin{lemma}\label{iso} Let $\phi:\,C\to C'$ be a twist such that $\phi\circ^{\sigma}\phi^{-1}=\xi_{\sigma}$. Then, the following $k-$vector spaces are isomorphic as $G_k-$modules: 
$$
\Omega_{L}^{1}(C)_{\xi}^{\operatorname{Gal}(L/k)}\simeq \Omega^{1}(C').
$$
\end{lemma}
\begin{proof} We claim that the map $\Omega_{L}^{1}(C)_{\xi}^{\operatorname{Gal}(L/k)}\rightarrow \Omega^{1}(C'):\,\omega\rightarrow \phi^{*}(\omega)$ is an isomorphism of  $G_k-$modules. The only non-trivial fact is the surjectivity. But this is a consequence of the equivalent result for function fields. 
Recall that the function field $k(C')$ may be reinterpreted as the fixed field $\overline{k}(C)_{\xi}^{G_{k}}$ where the action of the Galois group $G_{k}$ on $\overline{k}(C)$ is twisted by $\xi$ according to $f_{\xi}^{\sigma}:=f\cdot\xi_{\sigma}$ \cite[X.$2$]{Sil}.
\end{proof} 

We identify the previous vector spaces via an isomorphism as in Lemma \ref{iso}, so, for explicit computations, we can use
\begin{align}\label{CalcBase}
\Omega^{1}(C')=\bigcap_{\sigma\in\operatorname{Gal}(L/k)}\operatorname{Ker}(\rho_{\xi}(\sigma)-\operatorname{Id}).\end{align}

Consider the canonical morphism and the canonical model $\phi_K:\,C\rightarrow \mathcal{C}\subset \mathbb{P}^{g-1}$ given  by the basis $\{\omega_1,...,\omega_g\}$ of $\Omega^1(C)$. Let 
$$
\mathcal{C}:\:\left\{F_{h}(\omega_1,...,\omega_g)=0\right\}_h
$$
be a set of equations defining the canonical model in $\mathbb{P}^{g-1}$.
Let $\{\sum_{i=1}^{g}\mu_{j}^{i}\omega_{i}\}_j$ be a basis of $\Omega_{L}^{1}(C)_{\xi}^{\operatorname{Gal}(L/k)}$. We can then take a basis $\omega_{j}' =\sum_{i=1}^{g}\mu_{j}^{i}\omega_{i}$ of $\Omega^{1}(C')$ via an isomorphism as in Lemma \ref{iso}. Thus, we can write
$$
\omega_{i}=\sum_{j=1}^{g}\eta_{j}^{i}\omega_{j}'
$$
for some $\eta_{j}^{i}\in L$. We then obtain equations for the canonical model $\mathcal{C}'$, given by the basis $\{\omega_{j}' \}$, of the twist $C'$ via the substitution
$$
\mathcal{C}':\:\left\{F_{h}(\sum_{j=1}^{g}\eta_{j}^{1}\omega_{j}',...,\sum_{j=1}^{g}\eta_{j}^{g}\omega_{j}')=F'_{h}(\omega'_1,...\omega'_g)=0\right\}_h.
$$
Notice that the projective matrix $\eta=(\eta^{i}_{j})_{ij}$ defines an isomorphism of canonical models $\eta:\,\mathcal{C}'\rightarrow \mathcal{C}$, and that $\eta\cdot^{\sigma}\eta^{-1}=(\xi_{\sigma}^{*})^{-1}$. In general, on a canonical models level, any morphism of curves is given by a matrix, since a morphism of curves induces a linear morphism on the regular differential vector spaces.

\begin{remark} Notice that, for non-hyperelliptic curves, Lemma \ref{iso} is equivalent to prove that the dimension of $\Omega_{L}^{1}(C)_{\xi}^{\operatorname{Gal}(L/k)}$ is equal to $n$, that is, to find a matrix in $\eta\in\text{GL}_{g}(L)$ such that $\eta\cdot^{\sigma}\eta^{-1}=(\xi_{\sigma}^{*})^{-1}$. But this is a consequence of Hilbert $90$th Problem since $(\xi^{*})^{-1}\in\text{H}^1(\text{Gal}(L/k),\text{GL}(\Omega^{1}_{L}(C)))$.
\end{remark}

\section{Description of the method}
\label{13}

Let $C$ be a smooth non-hyperelliptic genus $g$ curve defined
over a perfect field $k$. Assume that its automorphism group $\operatorname{Aut}(C)$ is known and let us denote by $K$ the minimal field where $\operatorname{Aut}(C)$ is defined. We now proceed to describe a method for computing the set of twists of the curve $C$. In each step, we will compute different things:

\subsection*{Step 1: a canonical model} Firstly, we take a basis of $\Omega^{1}(C)$, and via this basis we
obtain a canonical model $\mathcal{C}/k$ as the image of the canonical morphism $C\hookrightarrow\mathbb{P}^{g-1}$. Again, the existence of a canonical divisor defined over $k$ implies that we can take the canonical model $\mathcal{C}$ also defined over $k$. Hence, $\mathcal{C}$ and
$C$ belong to the same class in $\mathrm{Twist}_k\left(\mathcal{C}\right)$
and $\mathrm{Twist}_k\left(\mathcal{C}\right)=\mathrm{Twist}_k\left(C\right)$.

In addition, the automorphisms group $\mathrm{Aut}\left(\mathcal{C}\right)$ 
can be viewed in a natural way as a subgroup of $\operatorname{PGL}_{g}\left(K\right)$ (via the induced automorphism in $\mathbb{P}^{g-1}$ by the canonical morphism).
Indeed, we can see it as a subgroup of $\operatorname{GL}_{g}\left(K\right)$ if we look at its action on $\Omega^1(C)\otimes_{k}K$ as a $K-$vector space.  Furthermore, any isomorphism $\phi:\,\mathcal{C}'\rightarrow \mathcal{C}$ can be also viewed as a matrix in $\operatorname{PGL}_{g}\left(\bar{k}\right)$.

\subsection*{Step 2: the set $\operatorname{Twist}_k(C)$} We will first compute the set $\widetilde{\mathrm{Hom}}\left(G_{k},\Gamma\right)$.
From this set, we will compute $\mathrm{H}^{1}\left(G_{k},\mathrm{Aut}\left(\mathcal{C}\right)\right)$ via the correspondence in Theorem \ref{main}.

Given an element $\Psi\in\widetilde{\mathrm{Hom}}\left(G_{k},\Gamma\right)$,
let $L$ be its splitting field. We have the following isomorphisms: $\Psi(G_K)\simeq\operatorname{Gal}(L/K)$ and  $\Psi(G_k)\simeq\operatorname{Gal}(L/k)$. Hence, we can see $\Psi$
as a proper solution to the Galois embedding problem
\begin{align*}
\xymatrix{
        &            &                   & G_k \ar@{-->>}_-{\Psi}[dl] \ar@{->>}[d]   &    \\
1 \ar[r] &\Psi(G_K) \ar[r]   & \Psi(G_k) \ar[r]  & \operatorname{Gal}(K/k)  \ar[r] &  1
} 
\end{align*}
As it was noticed in Section \ref{11}, we have  $\mathrm{Gal}\left(L/k\right)\simeq\mathrm{Image}\left(\Psi\right)\subseteq\Gamma$
and $\mathrm{Gal}\left(L/K\right)\simeq\Psi\left(G_{K}\right)\subseteq\mathrm{Aut}\left(\mathcal{C}\right)\rtimes\left\{ 1\right\} $.
Hence, we can break the computation of $\widetilde{\mathrm{Hom}}\left(G_{k},\Gamma\right)$, i.e., the solutions (proper or not) to the Galois embedding problem (\ref{gep}), into the computation of the proper solutions to some Galois embedding problems attached to a pair $(G,H)$ as follows
\begin{align}
\xymatrix{
        &            &                   & G_k \ar@{-->>}_-{\Psi}[dl] \ar@{->>}[d]   &    \\
1 \ar[r] &H \ar[r]   & G \ar[r]  & \operatorname{Gal}(K/k)  \ar[r] &  1
} \label{pGH}
\end{align}
where we consider all the pairs $(G,H)$ such that $G\subseteq\Gamma$,
$H=G\cap\mathrm{Aut}\left(C\right)\rtimes\left\{ 1\right\} $ and
$\left[G:H\right]=\left|\mathrm{Gal}\left(K/k\right)\right|$ (up to
conjugacy by elements $\left(\varphi,1\right)\in\Gamma$).

Every proper solution to a Galois embedding problem (\ref{pGH}) can be lifted to a solution to
the Galois embedding problem (\ref{gep}). 

Notice that the same field
$L$ can appear as the splitting field of more than one solution $\Psi$ corresponding to a pair $(G,H)$. This is because given an automorphism $\alpha$ of $\operatorname{Gal}(L/k)$ that leaves $\operatorname{Gal}(K/k)$ fixed, $\alpha\Psi$ is other solution that has $L$ as splitting field. Two such solutions are equivalent if and only if there exists $\beta\in \operatorname{Aut}(\mathcal{C})$ such that $\alpha\Psi=\beta\Psi\beta^{-1}$. So, the number of non-equivalent solutions with splitting field $L$ and $\Psi(\operatorname{G}_k)=G$ is the cardinality $n_{(G,H)}$ of the group \cite{Carp}:
\begin{align}
\mathrm{Aut}_{2}\left(G\right)/\operatorname{Inn}_{G}\left(\mathrm{Aut}\left(\mathcal{C}\right)\rtimes\left\{ 1\right\} \right),\label{numero}
\end{align}
where $\mathrm{Aut}_{2}\left(G\right)$ is the group of automorphisms
of $G$ such that leave the second coordinate invariant and $\operatorname{Inn}\left(\mathrm{Aut}\left(\mathcal{C}\right)\rtimes\left\{ 1\right\} \right)$
is the group of inner automorphisms of $\mathrm{Aut}\left(\mathcal{C}\right)\rtimes\left\{ 1\right\} $
lifted in the natural way to $\mathrm{Aut}\left(G\right)$. 

We can then divide this step in two:

\subsubsection*{Step 2a: computing the pairs $(G,H)$} The pairs $(G,H)$ and the number $n_{(G,H)}$ defined above, can be, for example, computed with magma \cite{magma} (c.f. \cite[Appendix]{thesis} for an implemented code). 

\subsubsection*{Step 2b: computing the proper solutions to the Galois embedding problems (\ref{pGH})} The solutions should be computed case-by-case for each pair $(G,H)$. If $k$ is a finite field this is know how to be done (e.g. \cite[Theorem $1.1$]{Ser}), and the method described in this paper becomes then an algorithm. Unfortunately, if $k$ is a number field there is not known systematical method for solving these problems

Next proposition, that is a generalization of \cite[Lemma $9.6$]{cox} for $q=3$, will be useful for solving some of these Galois embedding problems.

\begin{proposition}\label{q} Let be $q=p^r$, where $p$ is a prime number, let $k$ be a number field, and let $\zeta$ be a fixed $q$-th primitive root of the unity in $\bar{k}$. We denote $K=k(\zeta)$ and we assume $\left[  k(\zeta ):k\right] =p^{r-1}(p-1)$. Let us define $G_q:=\mathbb{Z}/q\mathbb{Z}\rtimes (\mathbb{Z}/q\mathbb{Z})^*$ where the action of $(\mathbb{Z}/q\mathbb{Z})^*$ on $\mathbb{Z}/q\mathbb{Z}$ is given by the multiplication rule $(a,b)(a',b')=(a+ba',bb')$. Let us consider the Galois embedding problem:
\[
\xymatrix{
        &            &                   & G_k  \ar@{->>}[d]_-{\pi}    &    \\
1 \ar[r] &\mathbb{Z}/q\mathbb{Z} \ar[r]   & G_q \ar[r] & (\mathbb{Z}/q\mathbb{Z})^*  \ar[r] &  1
}, 
\]
where the horizontal morphisms are the natural ones, and the projection $\pi$ is given by $\pi(\sigma)=(0,b)$ if $\sigma(\zeta)=\zeta^{b}$. Then, the splitting fields of the proper solutions to this Galois embedding problem are of the form $L=K(\sqrt[q]{m})$ where $m\in \mathcal{O}_k$ is an integer that is not a $p$-power. Moreover, every such field is the splitting field for a solution $\Psi$ to the above Galois embedding problem.
\end{proposition}

\begin{proof} Notice first that there exist proper solutions $\Psi$ to the Galois embedding problem. Given a field $L=K(\sqrt[q]{m})$ with $m\in k$ and not a $p$-power, there is a natural isomorphism $\operatorname{Gal}(L/k)\simeq G_q$ compatible with the projection $G_q \rightarrow (\mathbb{Z}/q\mathbb{Z})^*$. The natural projection $G_k\twoheadrightarrow \operatorname{Gal}(L/k)$ then provides a solution to the Galois embedding problem above.

 Now, let $\Psi$ be any proper solution to the problem, and let us denote by $L$ its splitting field. Let $G$ be the subgroup of  $G_q$ that contains all the elements of the form $(0,b)$, and let $\sigma\in G_k$ be such that $\Psi(\sigma)=(1,1)$.

Let $\alpha$ be a primitive element of the extension $L^G/k$ that moreover is an algebraic integer. We then have that $L=K(\alpha)$. This is because $\left[K:k\right]=p^{r-1}(p-1)$, $\left[L^G:k\right]=q$ and $L^G\cap K=k$. Let us now define for $i=0,1,...,q-1$ the numbers:
$$
u_i=\alpha+\zeta^i\sigma^{-1}(\alpha)+\zeta^{2i}\sigma^{-2}(\alpha)+...+\zeta^{(q-1)i}\sigma^{-(q-1)}(\alpha).
$$
Then $\sigma(u_i)=\zeta^{i}u_i$ and for any $\tau\in G_k$ such that $\Psi(\tau)=(0,b)$ we have $\Psi(\tau\sigma^{j})=(0,b)(j,1)=(bj,1)(0,b)=\Psi(\sigma^{bj}\tau)$, so $\tau(u_i)=u_i$ . In particular, we have that $u_0,\,u_1^q,...,\,u_{q-1}^q\in\mathcal{O}_k$. Hence, if $u_j\neq 0$ for some $j>0$, we have that $L=K(u_j)$, since $L^G=k(u_j)$. So, if we put $m=u_j^q\in \mathcal{O}_k$, we get $L=K(\sqrt[q]{m})$. Otherwise, that is, if $u_1=u_2=...=u_{q-1}=0$, then $u_0=u_0+u_1+...+u_{q-1}=q\alpha\in \mathcal{O}_k$, what is a contradiction with $\alpha$ being a primitive element of the extension $L^G/k$.

\end{proof}

For each proper solution $\Psi$ to a Galois embbeding problem (\ref{pGH}) attached to a pair $(G,H)$, we trivially compute the corresponding cocycle $\xi$ via the correspondence between the sets (\ref{e1}) and (\ref{e2}) in Theorem \ref{main}.

\subsection*{Step 3: Equations} We want to compute equations for a twist corresponding to a given cocycle $\xi$.
For this purpose we use the method explained in Section \ref{12}. Computing equations for a twist turns out to be equivalent to computing an isomorphism $\phi:\,C'\rightarrow C$, that is, to explicitly computing the inverse map to the correspondence between sets (\ref{e3}) and (\ref{e1}) in Theorem \ref{main}.

\section{An example}
\label{14}
In order to illustrate the method, we will apply it to the smooth non-hyperelliptic genus $6$ curve which admits the affine plane model
$$
C:\,x^7-y^3-1=0.
$$
As the point at infinity is singular, the projectivization of this plane model is not smooth. However, there is a unique curve, up to $\mathbb{Q}$-isomorphism, which is smooth and birationally equivalent to $C$. So, they have the same function field. We will apply the method for this smooth curve, which is non-hyperelliptic and has genus equal to $6$.
\subsection*{Step 1} First, we must find a canonical model by the usual procedure: finding a basis of holomorphic differentials. Let us call $X=x/z$ and $Y=y/z$. One has
$$
\operatorname{div}(X)=(0:-1:1)+(0:-\zeta_3:1)+(0:-\zeta_3^2:1)-3(0:1:0)=P_1+P_2+P_3-3\infty,
$$
$$
\operatorname{div}(Y)=Q_1+Q_2+Q_3+Q_4+Q_5+Q_6+Q_7-7\infty,
$$
where $Q_i=(\zeta_7^i:0:1)$. Then, $dX$ is an uniformizer for all points except for the $Q_i$'s, because the tangent space to the curve at these points have equation $X-\alpha=0$ for some $\alpha\in\bar{k}$. Then, for the points $Q_i$'s we have to use the expression
$$
dX=-\frac{3y^2}{7x^6}dY
$$
Thus, by \cite[Proposition $4.3$]{Sil}, we finally get
$$
\operatorname{div}(dX)=2(Q_1+Q_2+Q_3+Q_4+Q_5+Q_6+Q_7)-4\infty.
$$
We obtain the following basis of holomorphic differentials:
$$
\omega_1=\frac{dX}{Y^2},\,\,\omega_2=\frac{XdX}{Y^2},\,\,\omega_3=\frac{X^2dX}{Y^2},\,\,\omega_4=\frac{dX}{Y},\,\,\omega_5=\frac{X^3dX}{Y^2},\,\,\omega_6=\frac{XdX}{Y}.
$$

We consider the rational map
$$C\rightarrow\mathbb{P}^5:\,(x:y:z)\rightarrow (z^3:xz^2:x^2z:yz^2:x^3:xyz)$$
The ideal of the image of this map clearly contains the homogeneous polynomials:
$$
\begin{aligned}
f_1&=\omega_1\omega_6-\omega_2\omega_4,&f_2&=\omega_{2}^{2}-\omega_1\omega_3,&
f_3&=\omega_2\omega_3-\omega_1\omega_5,&f_4&=\omega_2\omega_5-\omega_{3}^{2}\\
f_5&=\omega_2\omega_6-\omega_3\omega_4,&f_6&=\omega_3\omega_6-\omega_4\omega_5,&
f_7&=\omega_{4}^{3}-\omega_{3}^{2}\omega_5+\omega_{1}^{3},&
f_8&=\omega_{5}^{3}-\omega_{4}\omega_{6}^{2}-\omega_{1}\omega_{2}^{2}.
\end{aligned}
$$
Now, we claim that the ideal generated by these polynomials gives a smooth curve. To see this, note that, if $\omega_1\neq 0$, the deshomogenization of this ideal with respect to $\omega_1$ gives the affine curve $C$. Now, we isolate from $f_2$ and $f_3$ the variables $\omega_3$ and $\omega_5$ and we plug them into $f_7$. Therefore, $C$ is birationally equivalent to $\mathcal{C}\cap\{ \omega_1\neq 0\}$. Next, if $\omega_1=0$, then the vanishing locus of $f_2,f_4,f_7,f_8$ is the point $(0:0:0:0:0:1)$. To check that $\mathcal{C}$ is non-singular at this point we consider the partial derivatives of the polynomials: $f_1,f_5,f_6,f_8$. Thus, $\mathcal{C}$ is a canonical model of the initial smooth non-hyperelliptic genus $6$ curve.

The automorphism group $\operatorname{Aut}(C)$ is generated by the automorphisms \cite{Swi}:
$$
(x:y:z)\rightarrow (x:\zeta_3y:z)\,\operatorname{and}\, (x:y:z)\rightarrow (\zeta_7x:y:z).
$$
Then, the automorphism group of the canonical model $\mathcal{C}$ is generated by the matrices in $\operatorname{PGL}_{6}\left(\bar{\mathbb{Q}}\right)$:
\small
$$
r=\left(\begin{array}{cccccc} \zeta_{3}&0&0&0&0&0\\ 0&\zeta_3&0&0&0&0\\0&0&\zeta_{3}&0&0&0\\0&0&0&\zeta_{3}^{2}&0&0\\0&0&0&0&\zeta_3&0\\0&0&0&0&0&\zeta_{3}^{2}\end{array}\right),\, s=\left( \begin{array}{cccccc} \zeta_{7}&0&0&0&0&0\\ 0&\zeta_{7}^{2}&0&0&0&0\\0&0&\zeta_{7}^{3}&0&0&0\\0&0&0&\zeta_7&0&0\\0&0&0&0&\zeta_{7}^{4}&0\\0&0&0&0&0&\zeta_{7}^{2}\end{array}\right).
$$
\normalsize

\subsection*{Step 2a} Let $k$ be a number field and consider the curve $\mathcal{C}/k$. We want to compute its twists over $k$. Let $K=k(\zeta_7,\,\zeta_3)$ and assume that $\left[K:k\right]=12$. Then, we compute using MAGMA the following possibilities for the pairs $(G,H)$: 

\vskip 0.5truecm

\begin{center}
\begin{tabular}{|c|c|c|c|c|}
\hline
 & $ \operatorname{ID}(G) $ \rule[0.4cm]{0cm}{0cm} & $ \operatorname{ID}(H)$ & $\operatorname{gen}(H)$ &$n_{(G,H)}$\tabularnewline
\hline
$1$ & $<12,5>$ & $<1,1>$ & $1$ & $1$\tabularnewline
\hline
$2$ & $<36,12>$ & $<3,1>$ & $r$ & $2$\tabularnewline
\hline
$3$ & $<84,7>$ & $<7,1>$ & $s$& $6$ \tabularnewline
\hline
$4$ & $<252,26>$ & $<21,2>$ & $r,\,s$ & $12$\tabularnewline
\hline

\end{tabular}
\end{center}

\vspace{3mm}
The fourth column in this table exhibits generators of the group $H$. In all the cases $G$ is the group generated by the elements $(g,1)$, for $g$ in $H$, together with the elements $(1,\tau_1)$ and $(1,\tau_2) $, where $\tau_1$ is the element in $\operatorname{Gal}(K/k)$ which sends $\zeta_3$ into $\zeta_3^2$ and $\zeta_7$ into $\zeta_7$, and $\tau_2$ is the element which sends $\zeta_3$ into $\zeta_3$ and $\zeta_7$ into $\zeta_7^3$. The fifth column exhibits the cardinality of the set in Formula (\ref{numero}) for each pair $(G,H)$.

\subsection*{Step 2b} Now, we have to find the proper solutions to the Galois embedding problems associated to each of the pairs $(G,H)$.

\begin{enumerate}
\item[1.] The first case is clear: $L=K$.
\item[2.] For the second one, note that $L=k(\zeta_7)M$, where $M/k$ is a solution to the Galois embedding problem in Proposition \ref{q} with $q=3$. Hence, $L=k(\zeta_3,\,\zeta_7,\,\sqrt[3]{m})$, for some $m\in\mathcal{O}_k$ that is not a $3$-power.
\item[3.] In this case, we can write $L=k(\zeta_3)M$, where $M/k$ is a solution to the Galois embedding problem in Proposition \ref{q} with $q=7$. Hence, $L=k(\zeta_3,\,\zeta_7,\,\sqrt[7]{n})$, for some $n\in\mathcal{O}_k$ that is not a $7$-power.
\item[4.] In the last case, $L=M_1M_2$, where $M_i/k$  is a solution to the Galois embedding problem in Proposition \ref{q} with $q=3,7$. Hence, $L=k(\zeta_3,\,\zeta_7,\,\sqrt[3]{m},\,\sqrt[7]{n})$, for some $m,n\in\mathcal{O}_k$, where $m$ is not a $3$-power and $m$ is not a $7$-power.
\end{enumerate}

\subsection*{Step 3} Fir each previous field $L$, we will compute equations of a twist that has $L$ as splitting field. The other twists, with splitting field $L$, will be then easily computed by considering symmetries. Let us consider a solution $\Psi$ (that is, a particular twist) to the Galois embedding problem with pair $(G,H)$ and splitting field $L$ by fixing an isomorphism between the group $H$ and the group $\text{Gal}(L/K)$: 
$$
(r,1):\,\sqrt[3]{m},\,\sqrt[7]{n}\rightarrow \zeta_3\sqrt[3]{m},\,\sqrt[7]{n},
$$
$$
(s,1):\,\sqrt[3]{m},\,\sqrt[7]{n}\rightarrow \sqrt[3]{m},\,\zeta_7\sqrt[7]{n}.
$$

Now, we compute equations for a twist in each case:
\begin{enumerate}
\item[1.] Clearly, this solution gives us the trivial twist, so we have the curve $\mathcal{C}/k$.
\item[2.] The correspondence between the sets (\ref{e1}) and (\ref{e2}) gives us the cocycle given by $\xi_{\tau_1}=1$, $\xi_{\tau_2}=1$ and $\xi_{(r,1)}=r$. If we take the basis of $\Omega_L^1(\mathcal{C})$ given by $\left\{(a,b,c,i)\right\}:=\left\{ \sqrt[3]{m^a}\zeta_3^b\zeta_7^c\omega_i\right\}$ where $a,b\in\{0,1,2\}$, $c\in\{0,1,...,6\}$ and $i\in\{1,...,6\}$, we obtain the twisted action of $\operatorname{Gal}(L/k)$ on $\Omega_L^1(\mathcal{C})$ given in Section \ref{12}:

$$
\tau_1(a,b,c,i)=(a,2b,c,i),\,\tau_2(a,b,c,i)=(a,b,3c,i)
$$

$$
(r,1)(a,b,c,i)=\begin{cases}(a,a+b+2,c,i)&\text{if}\,\,\,i=4,6\\ (a,a+b+1,c,i)&\text{otherwise}\end{cases}
$$

Now, we use formula (\ref{CalcBase}) and get a basis of $\Omega^{1}(\mathcal{C}')\simeq\Omega_{L}^{1}(\mathcal{C})_{\xi}^{\operatorname{Gal}(L/k)}$ given by:
$$
\left\{\sqrt[3]{m^2}\omega_1,\,\sqrt[3]{m^2}\omega_2,\,\sqrt[3]{m^2}\omega_3,\,\sqrt[3]{m}\omega_4,\,\sqrt[3]{m^2}\omega_5,\,\sqrt[3]{m}\omega_6 \right\}.
$$
So we get the generators of the ideal defining the twist:
$$
\begin{aligned}
\omega_1\omega_6&-\omega_2\omega_4,&\omega_{2}^{2}&-\omega_1\omega_3,&
\omega_2\omega_3&-\omega_1\omega_5,&\omega_2\omega_5&-\omega_{3}^{2},\\
\omega_2\omega_6&-\omega_3\omega_4,&\omega_3\omega_6&-\omega_4\omega_5,&
m\omega_{4}^{3}-\omega_{3}^{2}&\omega_5+\omega_{1}^{3},&
\omega_{5}^{3}-m\omega_{4}\omega_{6}^{2}-&\omega_{1}\omega_{2}^{2}
\end{aligned}
$$
We obtain generators for the other solution $\Psi$ that has $L$ as splitting field by exchanging $m$ by $m^2$.
\item[3.] In this case, the correspondence between the sets (\ref{e1}) and (\ref{e2}) gives us the cocycle given by $\xi_{\tau_1}=1$, $\xi_{\tau_2}=1$ and $\xi_{(s,1)}=s$. If we take the basis of $\Omega_L^1(\mathcal{C})$ given by $\left\{(a,b,c,i)\right\}:=\left\{ \sqrt[7]{n^a}\zeta_3^b\zeta_7^c\omega_i\right\}$, where $a,c\in\{0,1,...,6\}$, $b\in\{0,1,2\}$ and $i\in\{1,...,6\}$, we obtain the twisted action of $\operatorname{Gal}(L/k)$ on it given in Section \ref{12}:

$$
\tau_1(a,b,c,i)=(a,2b,c,i),\,\tau_2(a,b,c,i)=(a,b,3c,i)
$$

$$
(r,1)(a,b,c,i)=\begin{cases}(a,a+b+1,c,i)&\text{if}\,\,\,i=1,4\\ (a,a+b+2,c,i)&\text{if}\,\,\,i=2,6\\ (a,a+b+3,c,i)&\text{if}\,\,\,i=3\\ (a,a+b+4,c,i)&\text{if}\,\,\,i=5\end{cases}
$$

Now, we use formula (\ref{CalcBase}) again and get a basis of $\Omega^{1}(\mathcal{C}')\simeq\Omega_{L}^{1}(\mathcal{C})_{\xi}^{\operatorname{Gal}(L/k)}$ given by
$$
\left\{\sqrt[7]{n^6}\omega_1,\,\sqrt[7]{n^5}\omega_2,\,\sqrt[7]{n^4}\omega_3,\,\sqrt[7]{n^6}\omega_4,\,\sqrt[7]{n^3}\omega_5,\,\sqrt[7]{n^5}\omega_6 \right\}.
$$
Then, we get the set of generators of the ideal defining the twist:
$$
\begin{aligned}
\omega_1\omega_6&-\omega_2\omega_4,&\omega_{2}^{2}&-\omega_1\omega_3,&
\omega_2\omega_3&-\omega_1\omega_5,&\omega_2\omega_5&-\omega_{3}^{2},\\
\omega_2\omega_6&-\omega_3\omega_4,&\omega_3\omega_6&-\omega_4\omega_5,&
\omega_{4}^{3}-n\omega_{3}^{2}&\omega_5+\omega_{1}^{3},&
n\omega_{5}^{3}-\omega_{4}\omega_{6}^{2}-&\omega_{1}\omega_{2}^{2}
\end{aligned}
$$
We compute generators for the other solutions $\Psi$ that have splitting field equal to $L$ by exchanging $n$ by $n^2, \,n^3,\,n^4,\,n^5,\,n^6$.
\item[4.] In the last case, we have the cocycle given by $\xi_\tau=1$, $\xi_{(r,1)}=r$ and $\xi_{(s,1)}=s$. We take the basis of $\Omega_L^1(\mathcal{C})$ given by $\left\{(a,b,c,d,i)\right\}:=\left\{ \sqrt[3]{m^a}\sqrt[7]{n^b}\zeta_3^c\zeta_7^d\omega_i\right\}$ where $a,c\in\{0,1,2\}$,  $b,d\in\{0,...,6\}$ and $i\in\{1,...,6\}$, and we consider on $\Omega_L^1(\mathcal{C})$ the twisted action of $\operatorname{Gal}(L/k)$ given in Section \ref{12}.
Thus, formula (\ref{CalcBase}) provides a basis of $\Omega^{1}(\mathcal{C}')\simeq\Omega_{L}^{1}(\mathcal{C})_{\xi}^{\operatorname{Gal}(L/k)}$ given by:
$$
\left\{\sqrt[3]{m^2}\sqrt[7]{n^6}\omega_1,\,\sqrt[3]{m^2}\sqrt[7]{n^5}\omega_2,\,\sqrt[3]{m^2}\sqrt[7]{n^4}\omega_3,\,\sqrt[3]{m}\sqrt[7]{n^6}\omega_4,\,\sqrt[3]{m^2}\sqrt[7]{n^3}\omega_5,\,\sqrt[3]{m}\sqrt[7]{n^5}\omega_6 \right\}.
$$
Then, we get the set of generators of the ideal defining the twist
$$
\begin{aligned}
\omega_1\omega_6&-\omega_2\omega_4,&\omega_{2}^{2}&-\omega_1\omega_3,&
\omega_2\omega_3&-\omega_1\omega_5,&\omega_2\omega_5&-\omega_{3}^{2},\\
\omega_2\omega_6&-\omega_3\omega_4,&\omega_3\omega_6&-\omega_4\omega_5,&
m\omega_{4}^{3}-n\omega_{3}^{2}&\omega_5+\omega_{1}^{3},&
n\omega_{5}^{3}-m\omega_{4}\omega_{6}^{2}-&\omega_{1}\omega_{2}^{2}
\end{aligned}
$$
We compute generators for the other solutions $\Psi$ that have $L$ as splitting field by exchanging $m$ and $n$ by $m,\,m^2$ and $n,\,n^2, \,n^3,\,n^4,\,n^5,\,n^6$.
\end{enumerate}

We can summarize these results as follows: 
\begin{proposition} The twists of the curve $\mathcal{C}/k$ defined above where $k$ is a number field such that $\left[k(\zeta_{21}):k\right]=12$, are in one-to-one correspondence with the curves given by the ideals generated by the following  homogeneous polynomials:
$$
\begin{aligned}
\omega_1\omega_6&-\omega_2\omega_4,&\omega_{2}^{2}&-\omega_1\omega_3,&
\omega_2\omega_3&-\omega_1\omega_5,&\omega_2\omega_5&-\omega_{3}^{2},\\
\omega_2\omega_6&-\omega_3\omega_4,&\omega_3\omega_6&-\omega_4\omega_5,&
m\omega_{4}^{3}-n\omega_{3}^{2}&\omega_5+\omega_{1}^{3},&
n\omega_{5}^{3}-m\omega_{4}\omega_{6}^{2}-&\omega_{1}\omega_{2}^{2}
\end{aligned}
$$
where $m\in\mathcal{O}_{k}^{*}/(\mathcal{O}_{k}^{*})^3$ and $n\in\mathcal{O}_{k}^{*}/(\mathcal{O}_{k}^{*})^7$.
Equivalently, we can consider the (singular) plane models
$$
nx^7-my^3z^4-z^7=0.
$$
\end{proposition}

\end{document}